\documentclass[11pt,a4paper]{article}
\usepackage[T1]{fontenc}
\usepackage[utf8]{inputenc}
\usepackage{lmodern,amsmath,amssymb,amsthm,booktabs,array}
\usepackage[margin=27mm]{geometry}
\usepackage[colorlinks=true,linkcolor=blue,urlcolor=blue,citecolor=blue]{hyperref}
\usepackage{microtype}
\newtheorem{theorem}{Theorem}[section]
\newtheorem{lemma}[theorem]{Lemma}

\newtheorem{corollary}[theorem]{Corollary}
\newtheorem{conjecture}[theorem]{Conjecture}
\theoremstyle{definition}
\newtheorem{definition}[theorem]{Definition}
\newtheorem{remark}[theorem]{Remark}

\title{Hurwitz Equivalence of Reflection Factorizations\\ in Types $B$ and $D$}
\author{Patrick Wegener}
\date{}
\begin{document}
\maketitle
\begin{abstract}
\noindent A conjecture of Lewis asserts that two reflection factorizations of the same element in a finite Coxeter group $W$ belong to the same Hurwitz orbit if and only if they generate the same subgroup $H$ of $W$ and have the same multiset of $H$-conjugacy classes. We prove this conjecture for every finite Coxeter group all of whose irreducible
components are of type $A$, $B$ or $D$. Furthermore we prove a reduction, which reduces the conjecture to a finite problem for the exceptional types.
\end{abstract}

\section{Introduction}
A \emph{Coxeter group} $W$ comes with a distinguished generating set $S$ of involutions, the set of \emph{simple reflections}, and with the larger set
\[
T=\{gsg^{-1}\ :\ g\in W,\ s\in S\}
\]
of all \emph{reflections}. Since $T$ generates $W$, every element $g\in W$ can be written as a product of reflections. We call $(t_1, \dots, t_m) \in T^m$ a \emph{reflection factorization} of $g \in W$ if $g = t_1 \cdots t_m$. For further details and definitions concerning Coxeter groups, we refer to \cite{Humphreys}.

Let $\mathcal{B}_m$ be the braid group on $m$ strands, with standard generators
$\sigma_1,\dots,\sigma_{m-1}$. The \emph{Hurwitz action} of $\mathcal{B}_m$ on reflection factorizations is defined on generators by
\[
\sigma_i\cdot(t_1,\dots,t_m)=(t_1,\dots,t_{i-1},\ t_{i+1},\ t_{i+1}^{-1}t_it_{i+1},\ t_{i+2},\dots,t_m).
\]
\[
\sigma_i^{-1}\cdot(t_1,\dots,t_m)=(t_1,\dots,t_{i-1},\ t_it_{i+1}t_i^{-1},\ t_i,\ t_{i+2},\dots,t_m),
\]
For reflection factorization $\underline{t}$ and $\underline{r}$, we write $\underline{t}\sim \underline{r}$ if both tuples lie in the same orbit, and call them \emph{Hurwitz equivalent}. When applying a braid $\sigma$ to a reflection factorization, we will often simply refer to this as a \emph{Hurwitz move}.

\medskip
In  the  present  paper, we investigate the following conjecture (see \cite[Conjecture~1]{DL} and \cite[Section 5]{Lewis}).

\begin{conjecture}[Lewis] \label{conjecture}
Let $W$ be a finite Coxeter group and $g \in W$ an arbitrary element. Then two reflection factorizations of $g$ belong to  the same Hurwitz orbit if and only if they generate the same subgroup $H$ of $W$ and have the same multiset of $H$-conjugacy classes.
\end{conjecture}

The necessity of the conjecture follows directly from the definition of the Hurwitz action. We investigate sufficiency.

\medskip
A tuple $(t_1, \dots, t_m) \in T^m$ is called \emph{full in $W$} if its factors generate $W$, that is, $W = \langle t_1, \dots, t_m \rangle$. 

\begin{remark}
Say that a finite Coxeter group $W$ has \textbf{property (F)} if any two reflection factorizations of one and the same element of $W$ that are full in W and have the same multiset of $W$-conjugacy classes are Hurwitz equivalent.

A reflection subgroup $H$ is itself a finite Coxeter group with set of reflections $T \cap H$. A reflection factorization of an element $g\in W$ with generated subgroup $H$ is precisely a reflection factorization of $g$ that is full in $H$. 
Consequently, for a fixed finite Coxeter group $W$:

\textit{Conjecture~\ref{conjecture} holds for $W$ if and only if every reflection subgroup of $W$ has property (F).}

In particular, Conjecture 1.1 holds for all finite Coxeter groups if and only if all finite Coxeter groups have property (F).
\end{remark}

We call a reflection factorization $(t_1, \ldots, t_m)$ of an element $g$ of a Coxeter group $W$ \emph{reduced} if there is no reflection factorization of $g$ with less than $m$ factors. An element $g \in W$ is called \emph{parabolic quasi-Coxeter element} if there exists a reduced reflection factorization $(t_1, \dots, t_m)$ of $g$ such that $\langle t_1, \dots, t_m \rangle$ is a parabolic subgroup of $W$.
Lewis and Douvropoulos \cite{DL} proved the conjecture for the case that $g$ is a parabolic quasi-Coxeter element in a finite Coxeter group $W$. In particular this proves the conjecture for each element $g$ in a finite Coxeter group of type $A$, because in type $A$ each element of $W$ is a parabolic quasi-Coxeter element. The type $A$ case is also covered by a result of Kluitmann (see Lemma~\ref{lem:cl-symmetric} below). Further previous results were achieved for (quasi-)Coxeter elements \cite{BGRW, Bessis, LR, WY2}. We will show the following result.

\begin{theorem}\label{thm:main}
Conjecture~\ref{conjecture} holds for every finite Coxeter group all of whose irreducible components are of type $A$, $B$ or $D$.
\end{theorem}

Note that Conjecture~\ref{conjecture} has also been proven in the dihedral case by Berger \cite{Berger}. 

\medskip
\textbf{Outline of the paper.} In Section~\ref{sec:2} we prove a reduction of property (F). In Section~\ref{sec:classical-proof} we first show that property (F) holds in Coxeter groups of type $D$ and $B$ and then prove Theorem~\ref{thm:main}.

\medskip
\textbf{Notation.} We denote the symmetric group acting on $n$ letters by $\mathfrak{S}_n$. For a cycle $C = (i_1\,i_2\dots i_k)$ in $\mathfrak{S}_n$, its support is $\operatorname{supp}(C) = \{i_1, i_2, \dots, i_k\}$. Throughout we count fixed points as cycles of length $1$. If $g$ and $h$ are two elements of a group, we write $g^h = h^{-1}gh$ for conjugation. A conjugacy class of a group is sometimes referred to simply as a class. The support of a vector $v$, denoted by $\operatorname{supp}(v)$, is the set of indices where the vector's components are non-zero.

\medskip
\textbf{Disclosure of AI tools.} Generative AI tools, namely ChatGPT Astra and Claude Opus 5, were used during the exploratory and preparatory stages of this work. The author directed the mathematical strategy and proofs. AI assisted in several aspects like the formulation of auxiliary results and the elaboration of technical details, as well as consistency checks, and \LaTeX{} preparation. All AI-generated suggestions were independently reviewed and verified by the author, who assumes full responsibility for the mathematical arguments, results, and final content of the paper.

\section{A finite length reduction} \label{sec:2}
\begin{lemma}[{\cite[Lemma 2.5]{WY}}] \label{lem:pair}
Let $W$ be a Coxeter group with set of reflections $T$. If $\underline{t} = (t_1,\ldots ,t_m) \in T^m$ and $r \in T$, then for every $h \in\langle t_1,\ldots , t_m\rangle$,
\[
 (\underline{t} ,r,r)\sim(\underline{t}, h^{-1}rh,h^{-1}rh).
\]
In the latter equivalence the prefix $\underline{t}$ is restored exactly.
\end{lemma}

\begin{lemma} \label{lem:GenPrefix}
Let $W$ be a finite Coxeter group $W$ of rank $n$ such that every reflection generating set contains a generating subset of size $n$. Then every full reflection tuple is Hurwitz equivalent to a tuple of reflections
\[
 (u_1,\ldots,u_n,v_1,\ldots,v_d,r_1,r_1,\ldots,r_p,r_p),
\]
where $\langle u_1,\ldots,u_n\rangle=W$ and $d\leq n$.
\end{lemma}
\begin{proof}
Select $n$ entries generating $W$, in their original order. Move the first selected entry left to position one using Hurwitz moves, then the second to position two, and so on. 
The resulting prefix is the selected generating tuple. Put $q=u_1\cdots u_n$. The suffix has product $q^{-1}g$. Applying \cite[Corollary~1.4]{LR} only to this suffix yields the assertion. 
\end{proof}

\begin{definition}
Let $W$ be a Coxeter group with set of reflections $T$ and $(t_1,\dots, t_m) \in T^m$. For each $W$-conjugacy class $C$ of reflections write 
$$m_C((t_1,\dots,t_m)) = | \{ i : 1 \leq i \leq m, ~t_i \in C\}| $$ 
for the number of factors in $C$ and call $(m_C((t_1,\dots,t_m)))_C$ the \emph{class count vector} of $(t_1,\dots,t_m)$.
\end{definition}

\begin{definition}
Let $W$ be a finite irreducible Coxeter group of rank $n$ with set of reflections $T$. We call $R \subseteq T$ a \emph{minimal reflection generating set} if $\langle R \rangle = W$ and $|R| = n$.
Furthermore, define $b_C := |R \cap C|$.
\end{definition}

By \cite[Lemma~6.4]{WY} the definition of $b_C$ does not depend on the choice of $R$. Note that in the simply-laced types we have $(b_C)_C = (n)$, while in type $B_n$ we have $(b_C)_C = (1,n-1)$ since each minimal reflection generating set consists of one reflection in a short root and $n-1$ reflections in long roots.

\begin{theorem}\label{thm:box}
Let $W$ be an irreducible, non-dihedral finite Coxeter group of rank $n$ with set of reflections $T$. Suppose that for every $g \in W$ and every pair $\underline{t}, \underline{t}'$ of full reflection factorizations of $g$ in $W$ satisfying
\[
  m_C(\underline{t}) = m_C(\underline{t}') \le b_C + n
  \qquad \text{for every $W$-conjugacy class } C \subseteq T, 
\]
one has $\underline{t} \sim \underline{t}'$. Then the same conclusion holds for every pair of full reflection factorizations $\underline{t}, \underline{t}'$ of every $g \in W$ with $m_C(\underline{t}) = m_C(\underline{t}')$ for all $C$ (with no bound).
\end{theorem}

\begin{proof}
Fix $g\in W$ and a pair $\underline{t},\underline{t}'$ of full reflection factorizations of $g$ in $W$ with $m_C:= m_C(\underline{t})=m_C(\underline{t}')$ for every conjugacy class $C$. We write $N=\sum_C m_C$ for the common length of the two tuples. We argue by induction on $N$. The induction hypothesis is the assertion of the theorem for all $g$ and all such pairs of common length smaller than $N$. Note that fullness, i.e. $\langle\underline{t}\rangle= \langle\underline{t}'\rangle=H$, has to be maintained at every step.

Suppose first that $m_C\le b_C+n$ holds for every class $C$, that is, the pair $\underline{t},\underline{t}'$ fulfills the assumptions of the theorem, so $\underline{t}\sim\underline{t}'$ and there is nothing left to prove. This also provides the base of the induction: if $N\le n$ then $m_C\le N\le n\le b_C+n$ for every $C$, so every pair of common length at most $n$ also fulfills the assumption of the theorem.

From now on we may therefore fix a class $C_0$ with $m_{C_0}>b_{C_0}+n$. The aim is to produce full factorizations $\underline{a},\underline{b}$ of $g$ of common length $N-2$, with common class-count vector obtained from $(m_C)_C$ by lowering the $C_0$-entry by $2$, together with a single reflection $s\in C_0$ such that
\[
\underline{t}\sim(\underline{a},s,s)
\qquad\text{and}\qquad
\underline{t}'\sim(\underline{b},s,s).
\]

Since $W$ is irreducible, non-dihedral and of rank $n$, every reflection generating set of $W$ contains a generating subset of size $n$ by \cite[Lemma~13]{DL}, so Lemma~\ref{lem:GenPrefix} applies to both tuples. It yields
\[
\underline{t}\;\sim\;(u_1,\dots,u_n,\;v_1,\dots,v_d,\;r_1,r_1,\dots,r_p,r_p),
\]
where $\langle u_1,\dots,u_n\rangle=W$ and $d\le n$ with $q=u_1\cdots u_n$, and likewise for $\underline{t}'$. We call the tuples obtained in this way the \emph{normal forms} of $\underline{t}$ and $\underline{t}'$. Both normal forms are still full reflection factorizations of $g$ with the same class-count vector $(m_C)_C$ as before (as these properties are preserved by the Hurwitz action).

In the normal form of $\underline{t}$, the prefix $(u_1,\dots,u_n)$ is a reflection generating set of $W$ of size $n$, so it contains exactly $b_{C_0}$ entries of class $C_0$. The block $(v_1,\dots,v_d)$ has at most $n$ entries in total, hence at most $n$ of class $C_0$. Therefore the number of entries of class $C_0$ among the block $(r_1,r_1,\dots,r_p,r_p)$ is at least
\[
m_{C_0}-b_{C_0}-n>0 .
\]
The block $(r_1,r_1,\dots,r_p,r_p)$ consists of pairs of equal reflections, so an entry of class $C_0$ occurring there comes with its partner, which lies in $C_0$ as well. Hence some pair $(r_i,r_i)$ with $r:=r_i\in C_0$ occurs in the normal form of $\underline{t}$. By the same arguments some pair $(s,s)$ with $s\in C_0$ occurs in the normal form of $\underline{t}'$.

It is easy to see that an equal adjacent pair can be moved past the remaining factors by using Hurwitz moves without changing any of them. Moving the pair $(r,r)$ (resp. the pair $(s,s)$) to the very end of each tuple gives
\[
\underline{t}\sim(\underline{a},r,r),
\qquad
\underline{t}'\sim(\underline{b},s,s),
\qquad r,s\in C_0 ,
\]
where $\underline{a}$ and $\underline{b}$ consist of the remaining $N-2$ entries. Note that $\underline{a}$ still contains the generating prefix $(u_1,\dots,u_n)$, hence
$\langle \underline{a} \rangle=W$. The same holds mutatis mutandis for $\underline{b}$. The common class-count vector of $\underline{a}$ and $\underline{b}$ is $(m_C)_C$ with the $C_0$-entry lowered by $2$, and their common length is $N-2$.

The reflections $r$ and $s$ both lie in the class $C_0$, so $s=h^{-1}rh$ for some $h\in W$. As $\langle \underline{a} \rangle=W$, the element $h$ can be written as a word in the entries of $\underline{a}$. By Lemma~\ref{lem:pair} we have
\[
(\underline{a},r,r)\;\sim\;(\underline{a},h^{-1}rh,h^{-1}rh)=(\underline{a},s,s),
\]
with the prefix $\underline{a}$ restored entry by entry.

So far we have observed that $\underline{a}$ and $\underline{b}$ are full reflection factorizations of the same element $g\in W$, they have the same class-count vector, and their common length $N-2$ is smaller than $N$. The induction hypothesis therefore yields $\underline{a}\sim\underline{b}$. Performing the corresponding Hurwitz moves inside the first $N-2$ positions and leaving the last two untouched gives
$(\underline{a},s,s)\sim(\underline{b},s,s)$, whence
\[
\underline{t}\;\sim\;(\underline{a},r,r)\;\sim\;(\underline{a},s,s)\;\sim\;
(\underline{b},s,s)\;\sim\;\underline{t}' .
\]
This completes the induction and the proof.
\end{proof}

\begin{remark}
By Theorem~\ref{thm:box}, the problem of proving property (F) is reduced to a finite problem, with an explicit bound. In particular, for the finite irreducible Coxeter groups of exceptional type, we list the maximum length for which it is sufficient to check property (F).
\end{remark}

\begin{center}
\begin{tabular}{@{}lccc@{}}
\toprule
Type & Rank & Reflection classes & Sufficient maximum length\\
\midrule
$H_3$ & 3 & 1 & 6\\
$H_4$ & 4 & 1 & 8\\
$F_4$ & 4 & 2 & 12\\
$E_6$ & 6 & 1 & 12\\
$E_7$ & 7 & 1 & 14\\
$E_8$ & 8 & 1 & 16\\
\bottomrule
\end{tabular}
\end{center}
In type $B_n$ we obtain $3n$ as maximum length, while in type $D_n$ we obtain $2n$. In fact, we prove property (F) in the next section for types $B_n$ and $D_n$ without using these upper bounds.

\section{A proof for the classical families}
\label{sec:classical-proof}

In this section we prove property (F) for the finite Coxeter groups of types $B$ and $D$ and thereby conclude Theorem~\ref{thm:main}.

\subsection{Two elementary Hurwitz operations}

The following result is similar in spirit to Lemma \ref{lem:pair} and can be checked directly. 
\begin{lemma}
\label{lem:cl-pair}
Suppose a tuple of reflections contains a consecutive pair $(t,t)$, and a disjoint consecutive block $B$ whose product is $b$. There is a sequence of Hurwitz moves that replaces $(t,t)$ by $\left(t^b,t^b \right)$ and leaves every other factor, including the block $B$, unchanged.
\end{lemma}

\begin{lemma}[Kluitmann {\cite{KL88}}]
\label{lem:cl-symmetric}
The set of all transposition factorizations $(\tau_1, \tau_2, \dots , \tau_m)$ of fixed length $m$ of an element $\pi \in \mathfrak{S}_n$ such that $\langle \tau_1, \tau_2, \dots , \tau_m \rangle = \mathfrak{S}_n$ forms a single orbit under the Hurwitz action.
\end{lemma}

\subsection{Signed permutations and a projected normal form}

We realize the finite Coxeter group of type $B$ as the group of signed permutations. More precisely,
\[
 W_{B_n}=\mathbb F_2^n\rtimes\mathfrak S_n,
 \qquad
 (v,\pi)(w,\rho)=(v+\pi(w),\pi\rho),
\]
where $(v,\pi)$ acts as
$\operatorname{diag}((-1)^{v_1},\ldots,(-1)^{v_n})P_\pi$ and $\pi(w)$ is the the permuted vector
$\bigl(\pi(w)\bigr)_j = w_{\pi^{-1}(j)}$.
Write $e_i$ for the $i$th coordinate vector, and set
\[
 \delta_i=(e_i,1),\qquad
 t_{ij}^{a}=\bigl(a(e_i+e_j),(ij)\bigr)
 \quad(a\in\mathbb F_2,\ i\ne j).
\]
The $\delta_i$ are the \emph{short} reflections of $W_{B_n}$, and the $t_{ij}^{a}$ are its \emph{long} reflections. The subgroup $W_{D_n}$ consists of the pairs $(v,\pi)$ with $\sum_i v_i=0$. Its reflections are exactly the long reflections. The following identities will be used repeatedly:
\begin{align}
 t_{ij}^{a}t_{ij}^{b}
   &=\delta_i^{a+b}\delta_j^{a+b},
 \label{eq:cl-pair-product}\\
 \bigl(t_{ij}^{a}\bigr)^{\delta_u}
   &=t_{ij}^{a+u_i+u_j},
 \qquad \delta_u:=(u,1).
 \label{eq:cl-diagonal-conjugation}
\end{align}
In particular, $t_{ij}^0$ and $t_{ij}^1$ commute. Therefore every ordering of a sequence supported on the same underlying transposition $(i\, j)$ can be obtained by Hurwitz moves that merely swap neighboring factors. For $g=(v,\pi) \in W_{B_n}$ and a cycle $C$ of $\pi$, define its \emph{signed parity} by
\[
 p(C)=\sum_{j\in C}v_j\in\mathbb F_2.
\]
Call $C$ \emph{positive} if $p(C)=0$ and \emph{negative} if $p(C)=1$. An element of $W_{D_n}$ has an even number of negative cycles.

\begin{lemma}
\label{lem:cl-tree}
Let $C \in \mathfrak{S}_n$ be a cycle with support $V := \operatorname{supp}(C)$, $|V| = k \geq 2$, and let $T = (\tau_1,\dots,\tau_{k-1})$ be a reduced transposition factorization of $C$, say $\tau_m = (i_m\, j_m)$. Let $z = (v,C) \in W_{B_n}$ with $\operatorname{supp}(v) \subseteq V$ and $p(C) = \sum_{j \in V} v_j = 0$. Then there is exactly one labelling $a = (a_1,\dots,a_{k-1}) \in \mathbb F_2^{\,k-1}$ with
\[
  t^{a_1}_{i_1 j_1}\, t^{a_2}_{i_2 j_2}\cdots t^{a_{k-1}}_{i_{k-1} j_{k-1}} \;=\; z .
\]
\end{lemma}

\begin{proof}
Write $\mathbb F_2^V \subseteq \mathbb F_2^n$ for the vectors supported in $V$, and put
\[
  \Phi \colon \mathbb F_2^{\,k-1} \longrightarrow W_{B_n}, ~\alpha = (a_1,\dots, a_{k-1}) \mapsto t^{a_1}_{i_1 j_1}\cdots t^{a_{k-1}}_{i_{k-1} j_{k-1}} ,
\]
so the assertion is that $\Phi$ maps $\mathbb F_2^{\,k-1}$ bijectively onto the set
\[
  Z := \bigl\{\, (v,C) : \operatorname{supp}(v) \subseteq V,\ \textstyle\sum_{j\in V} v_j = 0 \,\bigr\}.
\]
Note that $\Phi(\alpha)$ has underlying permutation $\tau_1\cdots\tau_{k-1} = C$, and all
vectors involved are supported in $V$. Moreover the total sign parity
\[
  \varepsilon \colon W_{B_n} \to \mathbb F_2, \qquad \varepsilon\bigl((v,\pi)\bigr) := \sum_{j=1}^n v_j,
\]
is a group homomorphism. Since
$t^a_{ij} = \bigl(a(e_i+e_j),(ij)\bigr)$ has $\varepsilon(t^a_{ij}) = 2a = 0$, we get
$\varepsilon(\Phi(\alpha)) = 0$ for every $\alpha$. Hence $\Phi(\alpha) \in Z$.

By \cite[Lemma~3.10.2]{GR01} the graph $G(T) := \bigl(V, \{\,\{i_m,j_m\} : 1 \le m \le k-1\,\}\bigr)$ is a tree.
Define the $\mathbb F_2$-linear map
\[
  \partial \colon \mathbb F_2^V \longrightarrow \mathbb F_2^{\,k-1}, \qquad
  (\partial u)_m := u_{i_m} + u_{j_m} ,
\]
which labels the edge $\{i_m,j_m\}$ by $(\partial u)_m$. By \eqref{eq:cl-diagonal-conjugation} we have
$\bigl(t^{a_m}_{i_m j_m}\bigr)^{\delta_u} = t^{a_m + u_{i_m}+u_{j_m}}_{i_m j_m}$, that is,
\begin{equation}\label{eq:gauge}
  \Phi(\alpha)^{\delta_u} \;=\; \Phi(\alpha + \partial u) \qquad (\alpha \in \mathbb F_2^{\,k-1},\ u \in \mathbb F_2^V).
\end{equation}
We have $\partial u = 0$ if and only if $u_i = u_j$ for every edge $\{i,j\}$ of $G(T)$, that is, $\partial u = 0$ if and only if $u$ is constant on each connected component of $G(T)$. Since $G(T)$ is a tree, hence connected, we obtain that $\ker \partial = \{0, \mathbf 1_V\}$ has dimension $1$. The rank-nullity theorem yields $\dim \operatorname{im} \partial = k - 1$, hence $\partial$ is surjective.
In particular every labelling $\alpha$ is of the form $\alpha = \partial u$, and
\begin{equation}\label{eq:conj}
  \Phi(\alpha) \;=\; 
  \Phi(\partial u) \;=\; \Phi(0+\partial u)\; \stackrel{\eqref{eq:gauge}}{=} \; \Phi(0)^{\delta_u}
  \;=\; (u,1)(0,C)(u,1) \;=\; \bigl(u + C(u),\, C\bigr).
\end{equation}
Now consider the $\mathbb F_2$-linear map $\psi \colon \mathbb F_2^V \to \mathbb F_2^V$, $\psi(u) := u + C(u)$. Since $C$ acts transitively on $V$, we have $C(u) = u$ if and only if $u$ is constant on $V$, so
$\ker\psi = \{0,\mathbf 1_V\}$ and $\dim\operatorname{im}\psi = k-1$.
Moreover $\varepsilon(\psi(u)) = \sum_j u_j + \sum_j u_j = 0$, so $\operatorname{im}\psi$ is
contained in the subspace of vectors of coordinate sum $0$, which also has dimension $k-1$. Hence $\operatorname{im}\psi$ is exactly that subspace. Therefore, if $z = (v,C) \in Z$, there is $u \in \mathbb F_2^V$ with $u + C(u) = v$, and then $\alpha := \partial u$ satisfies $\Phi(\alpha) = z$ by \eqref{eq:conj}.

It remains to show that $\alpha$ is unique. Let $\alpha, \alpha' \in \mathbb F_2^{\,k-1}$ with $\Phi(\alpha) = \Phi(\alpha') = z$. By surjectivity of $\partial$, we may write $\alpha = \partial u$ and $\alpha' = \partial u'$, and by \eqref{eq:conj} we obtain $u + C(u) = u' + C(u')$, i.e.\
$\psi(u+u') = 0$. Hence we have $u + u' \in \ker(\psi) = \{0,\mathbf 1_V\} = \ker\partial$, which gives $0 = \partial(u+u') = \partial u + \partial u'$ and therefore
$\alpha' = \partial u' = \partial u = \alpha$.
\end{proof}

\begin{lemma}
\label{lem:cl-projection}
Let $\pi \in \mathfrak S_n$ have cycles $C_1,\dots,C_c$ with $c \ge 2$. Fix representatives $r_i \in \operatorname{supp}(C_i)$ and, for each $i$, a reduced transposition factorization $T_i$ of $C_i$ (which is empty if $|\operatorname{supp}(C_i)| = 1$). Put
$e := (r_1\,r_2)$ and $f_i := (r_1\,r_i)$ for $3 \le i \le c$. If
$\underline{\tau} = (\tau_1,\dots,\tau_m)$ is a transposition factorization of $\pi$ with
$\langle \tau_1,\dots,\tau_m\rangle = \mathfrak S_n$, then
\begin{equation} \label{eq:cl-normal-form}
  m = n + c - 2 + 2h \quad\text{for the integer } h := \tfrac12\bigl(m - n - c + 2\bigr) \ge 0,
\end{equation}
and $\underline{\tau}$ is Hurwitz equivalent to
\begin{equation}\label{eq:nf}
  \bigl(\underbrace{e,\dots,e}_{2h+2 \text{ entries}},\;
        f_3,f_3,\;\dots,\;f_c,f_c,\;
        T_1,\dots,T_c\bigr).
\end{equation}
\end{lemma}

\begin{proof}
We write $c(\sigma)$ for the number of cycles of a permutation $\sigma \in \mathfrak S_n$ (with fixed points included). For $0 \le l \le m$ we write $\pi_l := \tau_1\cdots\tau_l$ with $\pi_0 = \mathrm{id}$ and $\tau_l = (i_l \, j_l)$. Then
\[
  c(\pi_l) = c(\pi_{l-1} \tau_l) =
  \begin{cases}
    c(\pi_{-1}) - 1, & \text{if $i_l$ and $j_l$ lie in different cycles of $\pi_{l-1}$,}\\
    c(\pi_{l-1}) + 1, & \text{if $i_l$ and $j_l$ lie in the same cycle of $\pi_{l-1}$;}
  \end{cases}
\]
If we denote by $J$ (resp. $S$) the number of indices $l \in \{1,\dots,m\}$ such $i_l$ and $j_l$ lie in different cycles (resp. in the same cycle) of $\pi_{l-1}$, that is $J+S = m$, we obtain
\begin{equation} \label{equ:cnJS}
  c = c(\pi_m) = c(\pi_0) - J + S = n - J + S.
\end{equation}

\smallskip
For $0 < l \leq m$, we define the graph
\[
G_l = G\bigl(\tau_1,\dots,\tau_l \bigr) := \bigl([n],\ \{\,\{i_p,j_p\} : 1 \le p \le l\,\}\bigr).
\]
and define $G_0$ as the graph on vertex set $[n]$, where each vertex is isolated. The orbits of $\langle \tau_1, \dots, \tau_l\rangle$ on $[n]$ are the vertex sets of the connected components of $G\bigl(\tau_1,\dots,\tau_l \bigr)$. By \cite[Lemma~3.10.1]{GR01} on each connected component the generated subgroup induces the full symmetric group of that component. In particular $\langle\tau_1,\dots,\tau_l\rangle = \mathfrak S_n$ if and only if $G\bigl(\tau_1,\dots,\tau_l \bigr)$ is connected. 

Call $l$ a \emph{connecting} index if the edge $\{i_l,j_l\}$ joins two distinct connected components of $G\bigl(\tau_1,\dots,\tau_{l-1} \bigr)$. We claim that if $l$ is a connecting index, then $i_l$ and $j_l$ lie in different cycles of $\pi_{l-1}$. As $\pi_{l-1} \in \langle \tau_1,\dots,\tau_{l-1}\rangle$, each orbit of $\langle \pi_{l-1}\rangle$, that is each cycle of $\pi_{l-1}$, is contained in an orbit of $\langle \tau_1,\dots,\tau_{l-1}\rangle$, hence in a connected component of $G\bigl(\tau_1,\dots,\tau_{l-1} \bigr)$. If $l$ is connecting, then $i_l$ and $j_l$ lie in different connected components of $G\bigl(\tau_1,\dots,\tau_{l-1} \bigr)$, so they lie in different cycles of $\pi_{l-1}$.

The graph $G_0$ has $n$ connected components. Passing from $G_{l-1}$ to $G_l$ decreases the number of components by $1$ if $l$ is connecting and leaves it unchanged otherwise. As $\langle \tau_1, \dots, \tau_m \rangle = \mathfrak S_n$, the graph $G_m = G(\underline{\tau})$ is connected by. Hence there are exactly $n - 1$ connecting indices, and $J \ge n-1$.

\smallskip
Substituting \eqref{equ:cnJS} into the equation $J+S=m$ gives $m = 2J - (n - c)$, that is, $m = n + c - 2 + 2\bigl(J - n + 1\bigr)$, and $h := J - n + 1 \ge 0$ since $J \geq n-1$. Solving for $h$ shows $h = \frac12(m-n-c+2)$. In particular $h$ depends only on $m$, $n$ and $\pi$, not on the tuple $\underline{\tau}$.

We want to show that the tuple \eqref{eq:nf} is a transposition factorization of $\pi$ of length $m$ generating $\mathfrak S_n$. Denote it by $\underline{\nu}$. All factors $e$ and $f_i$ ($3 \leq i \leq c$) are transpositions and each $T_i$ has product $C_i$ with the $C_i$ pairwise disjoint, hence commute. So the product of $\underline{\nu}$ equals $C_1 \cdots C_c = \pi$. Since $T_i$ is reduced, it has $|\operatorname{supp}(C_i)| - 1$ entries, so $\underline{\nu}$ has
\[
  (2h+2) + 2(c-2) + \sum_{i=1}^{c}\bigl(|\operatorname{supp}(C_i)| - 1\bigr)
  = 2h + 2 + 2c - 4 + n - c = n + c - 2 + 2h = m
\]
entries. As we have seen in the proof of Lemma~\ref{lem:cl-tree}, the graph $G(T_i)$ is a tree on $\operatorname{supp}(C_i)$. In particular $\operatorname{supp}(C_i)$ lies in a single connected component of $G(\underline{\nu})$. The edges $\{r_1,r_2\}$ and $\{r_1,r_i\}$ ($3 \le i \le c$) then join these $c$ connected components into one, so $G(\underline{\nu})$ is connected on $[n]$, thus if follows by \cite[Lemma~3.10.1]{GR01} that $\underline{\nu}$ generates $\mathfrak S_n$.

Overall, we have seen that both $\underline{\tau}$ and $\underline{\nu}$ are transposition factorizations of $\pi$ of the same length $m$ generating $\mathfrak S_n$. By Lemma~\ref{lem:cl-symmetric} we obtain $\underline{\tau} \sim \underline{\nu}$.
\end{proof}

\subsection{Property (F) for type $D$}

\begin{lemma}[{\cite[Corollary~3.16]{DLM}}]
\label{lem:cl-d-parabolic}
If an element of $W_{D_n}$ has zero or two negative cycles, it is a
parabolic quasi-Coxeter element.
\end{lemma}

For a subset $V \subseteq [n]$ put
\[
  W_{B_V} := \bigl\{\, (v,\pi) \in W_{B_n} \ :\ \operatorname{supp}(v) \subseteq V
      \text{ and } \pi(j) = j \text{ for all } j \notin V \,\bigr\},
\]
the subgroup of those signed permutations that move no coordinate outside $V$. 
Restricting to the coordinates in $V$ identifies $W_{B_V}$ with $W_{B_{|V|}}$, and under this identification the reflections of $W_{B_n}$ lying in $W_{B_V}$ are exactly the $\delta_i$ with $i \in V$ and the $t^a_{ij}$ with $i,j \in V$, that is, the reflections of $W_{B_{|V|}}$. 
For $g = (v,\pi) \in W_{B_V}$ we write
\begin{equation} \label{eq:DefpV}
  p_V(g) := \sum_{j \in V} v_j \in \mathbb F_2
\end{equation}
for its signed parity on $V$. If $\pi$ restricted to $V$ is a single cycle $C$, then $p_V(g) = p(C)$.

\smallskip
We consider the projection 
$$\rho \colon W_{B_n} \to \mathfrak S_n,~(v,\pi) \mapsto \pi$$
and we apply it to tuples entrywise, writing
$\rho(\underline{t}) := (\rho(t_1),\dots,\rho(t_m))$, where $\underline{t} = (t_1, \dots, t_m)$. The map $\rho$ is equivariant with respect to the Hurwitz action, that is
\begin{equation}\label{eq:equivariance}
  \rho(\beta \cdot \underline{t}) = \beta \cdot \rho(\underline{t})
  \quad\text{for every braid } \beta .
\end{equation}
Recall also that every long reflection is $t^a_{ij}$ for a unique $a \in \mathbb F_2$ and a
unique pair $\{i,j\}$, so a tuple of long reflections is determined by its sequence of
underlying transpositions together with its sequence of labels.

\begin{theorem}
\label{thm:cl-d-full}
For $n\ge3$, two full reflection factorizations of an element of $W_{D_n}$ of the same length are Hurwitz equivalent.
\end{theorem}

\begin{proof}
First note that all reflections of $W_{D_n}$ are conjugate in $W_{D_n}$. The hypothesis of the theorem is therefore the hypothesis of Conjecture~\ref{conjecture}.

Writing $g = (v,\pi)$ with cycles $C_1,\dots,C_c$ of $\pi$ (fixed points included), the total sign parity gives
\begin{equation} \label{eq:totalsignparity}
\sum_{i=1}^{c} p(C_i) = \sum_{j=1}^{n} v_j = 0
\end{equation}
so the number of negative cycles of $g$ is even. If it is $0$ or $2$, then $g$ is a parabolic quasi-Coxeter element by Lemma~\ref{lem:cl-d-parabolic}, and the assertion is \cite[Theorem~18]{DL}. We may therefore assume that $g$ has at least four negative cycles. In particular we have $c \ge 4$.

Index the cycles so that $C_2$ and $C_3$ are negative, and put $V_i := \operatorname{supp}(C_i)$, so that $[n] = V_1 \sqcup \dots \sqcup V_c$. Fix once and for all representatives $r_i \in V_i$ and reduced transposition factorizations $T_i$ of $C_i$, and set $e := (r_1\,r_2)$ and $f_i := (r_1\,r_i)$ for $3 \le i \le c$. Let $m$ be the length of the given full reflection factorization and let
\[
  h := \tfrac12\,(m - n - c + 2),
\]
which by Lemma~\ref{lem:cl-projection} is a non-negative integer depending only on $m$, $n$ and $\pi$. All choices made so far depend only on $g$ and $m$, not on the tuples.

Call a tuple $\underline{t}$ of long reflections \emph{in normal position} if $\rho(\underline{t})$
is the tuple \eqref{eq:nf} of Lemma~\ref{lem:cl-projection}, i.e.\ if
\begin{equation}\label{eq:normalpos}
  \underline{t} = \bigl(
    \underbrace{t^{a_1}_{r_1r_2},\dots,t^{a_{2h+2}}_{r_1r_2}}_{e\text{-lift}},\;
    \underbrace{t^{b_3}_{r_1r_3},t^{b'_3}_{r_1r_3},\dots,t^{b_c}_{r_1r_c},t^{b'_c}_{r_1r_c}}_{f-\text{lift}},\;
    \underline{U}_1,\dots,\underline{U}_c
  \bigr),
\end{equation}
where $\underline{U}_i$ is a tuple of $|V_i| - 1$ long reflections whose underlying
transpositions are those of $T_i$, in the order prescribed by $T_i$. We call the first two blocks of \eqref{eq:normalpos} the \emph{bridge} and the remaining blocks the \emph{tree part}, and we call 
$$(a_1, \dots, a_{2h+2}, b_3, b'_3, \dots, b_c, b'_c)$$ together with the labels occurring in the
$\underline{U}_i$ the \emph{label vector} of $\underline{t}$. A tuple in normal position is determined by its label vector.

Now let $\underline{t}$ be any full reflection factorization of $g$ in $W_{D_n}$ of length $m$. Its entries are long reflections, so $\rho(\underline{t})$ is a transposition factorization of $\pi$, which generates $\mathfrak{S}_n$, because $\rho(W_{D_n}) = \mathfrak{S}_n$ and $\underline{t}$ generates $W_{D_n}$. By Lemma~\ref{lem:cl-projection} there is a braid $\beta$ such that $\beta \cdot \rho(\underline{t})$ equals the factorization in \eqref{eq:nf}, and by \eqref{eq:equivariance} the tuple $\beta \cdot \underline{t}$ is in normal position. Replacing $\underline{t}$ by $\beta \cdot \underline{t}$, we may and do assume from now on that $\underline{t}$ is in normal
position.

For $\underline{t}$ as in \eqref{eq:normalpos} put
\[
  q_e := |\{\,\mu : a_\mu = 1\,\}| \in \mathbb Z_{\ge 0},
  \quad
  d_2 := \sum_{\mu=1}^{2h+2} a_\mu \in \mathbb F_2,
  \quad
  d_i := b_i + b'_i \in \mathbb F_2 ~ (3 \le i \le c),
\]
so that $d_2 \equiv q_e \pmod 2$. We will show that $d_2,\dots,d_c$ are determined by $g$.

Let $D$ denote the product of the bridge and $U$ that of the tree part, so that $g = DU$.
Grouping the $e$-lifts into $h+1$ consecutive pairs and applying \eqref{eq:cl-pair-product} to each pair gives
\[
  t^{a_1}_{r_1r_2}\cdots t^{a_{2h+2}}_{r_1r_2}
  = \prod_{s=1}^{h+1} (\delta_{r_1}\delta_{r_2})^{a_{2s-1}+a_{2s}}
  = (\delta_{r_1}\delta_{r_2})^{q_e}
  = (\delta_{r_1}\delta_{r_2})^{d_2}.
\]
Identity \eqref{eq:cl-pair-product} applied to each pair in the $f$-lifts gives $t^{b_i}_{r_1r_i}t^{b'_i}_{r_1r_i} = (\delta_{r_1}\delta_{r_i})^{d_i}$.
Since $\delta_{r_1}\delta_{r_i} = \delta_{e_{r_1}+e_{r_i}}$ and $\delta_u\delta_w =
\delta_{u+w}$, we conclude
\begin{equation}\label{eq:D}
  D = \prod_{i=2}^{c}
       (\delta_{r_1}\delta_{r_i})^{d_i} = \delta_{u}, \qquad
  u = \sum_{i=2}^{c} d_i\,(e_{r_1} + e_{r_i})
    = \Bigl(\sum_{i=2}^{c} d_i\Bigr) e_{r_1} + \sum_{i=2}^{c} d_i\, e_{r_i}.
\end{equation}
In particular $D$ is an involution, so $D^{-1} = D$ and $U = Dg$.

Write $U_i$ for the product of $\underline{U}_i$. We have $U_i \in W_{B_{V_i}}$ with $\rho(U_i) = C_i$. Since $p_{V_i}$ vanishes on every long reflection, we have $p_{V_i}(U_i) = 0$.
As the $V_i$ are pairwise disjoint, $U = U_1 \cdots U_c = (w,\pi)$ where $w|_{V_i}$ is the sign
vector of $U_i$, thus
\begin{equation}\label{eq:Uparity}
  \sum_{j \in V_i} w_j = 0 \qquad (1 \le i \le c).
\end{equation}
From $g = DU = (u,1)(w,\pi) = (u + w, \pi)$ we obtain $v = u + w$, hence
\[
  p(C_i) = \sum_{j \in V_i} v_j \stackrel{\eqref{eq:Uparity}}{=} \sum_{j \in V_i} u_j .
\]
By \eqref{eq:D}, for $i \ge 2$ the only coordinate of $u$ inside $V_i$ is the one at $r_i$,
whereas for $i = 1$ it is the one at $r_1$. Therefore
\begin{equation}\label{eq:di}
  d_i = p(C_i) \quad (2 \le i \le c),
  \qquad\text{and}\qquad
  \sum_{i=2}^{c} d_i \stackrel{\eqref{eq:totalsignparity}}{=} p(C_1).
\end{equation}
In particular $d_2 = p(C_2) = 1$, so $q_e$ is odd, and $d_3 = p(C_3) = 1$, so the pair lifting $(f_3,f_3)$ has distinct labels and product $\delta_{r_1}\delta_{r_3} \ne 1$ by \eqref{eq:cl-pair-product}.

\smallskip
\noindent We note the following two things:
\begin{itemize}
    \item[(i)] Two long reflections with the same underlying transposition commute. Therefore, by using Hurwitz moves, the entries of the $e$-lift may be permuted arbitrarily, while any pair lifting $(f_i,f_i)$ may be interchanged. All other entries remain unchanged. 
    \item[(ii)] If $A= (t^a_{ij}, t^a_{ij})$ is a pair of equal entries of $\underline{t}$ and $B$ a disjoint block of consecutive entries with product $\delta_w$, then by Lemma~\ref{lem:cl-pair} there are Hurwitz moves which replace $A$ by $(t^{a+w_i+w_j}_{ij}, t^{a+w_i+w_j}_{ij})$ and leaves all other entries (including $B$) unchanged.
\end{itemize}
Both operations (i) and (ii) preserve the underlying transpositions of all entries, hence preserve normal position.

We claim that $\underline{t}$ can be transformed into a tuple in normal position with the same pairs in the $f$-lift and the same tree part, but whose $e$-lift has label vector $(0,1,0,\dots,0)$.

Suppose $q_e \ge 3$. By (i) we can permute the $e$-lift so that two entries with label $1$ are adjacent. As they are equal, (ii) applies with $A$ this pair and $B$ the pair $(t^{b_3}_{r_1r_3},t^{b'_3}_{r_1r_3})$ lifting $(f_3,f_3)$, whose product is $\delta_w$ with $w = e_{r_1} + e_{r_3}$. 
So both labels in the $e$-lift change from $1$ to $0$ and $q_e$ decreases by $2$. As $q_e$ is odd, we can repeat this process until $q_e = 1$. Furthermore, using (i), we can put the unique label-$1$ entry into the second position.

We claim the tuple may be further transformed so that, in addition, the pair $(t^{b_3}_{r_1r_i},t^{b'_3}_{r_1r_i})$ lifting $(f_i,f_i)$ has label vector $(0,d_i)$ for $3 \le i \le c$, the $e$-lift being unchanged.

We fix $i$. If $d_i = 1$, the two labels are $0$ and $1$, and by (i) we may interchange them if necessary to obtain $(0,1)$. If $d_i = 0$, the two labels are equal. If they are $(1,1)$, apply (ii) with $A$ this pair and $B$ the first two entries of the $e$-lift, which have labels $0$ and $1$ and hence have product $t^0_{r_1r_2}t^1_{r_1r_2} = \delta_w$ with $w = e_{r_1}+e_{r_2}$ by \eqref{eq:cl-pair-product}. Since $r_i \ne r_2$, we have $w_{r_1} + w_{r_i} = 1$, so both labels change from $1$ to $0$, giving $(0,0)$. In either case the $e$-lift and the other pairs in the $f$-lift remain unchanged, so the pairs may be treated one after the other. 

We have observed at the beginning of the proof, that the choices of $r_i$, $T_i$ and $h$ only depend on $g$ and $m$ and up to this point, the bridge has labels
\[
  (a_1,\dots,a_{2h+2}) = (0,1,0,\dots,0), \qquad (b_i, b'_i) = (0, d_i) \ \ (3 \le i \le c),
\]
which by \eqref{eq:di} therefore also depend only on $g$ and $m$. Hence so does $D$ by \eqref{eq:D}, thus also $U = Dg$. By \eqref{eq:di} the sign vector $w$ of $U = (w,\pi)$ satisfies $\sum_{j \in V_i} w_j = p(C_i) + d_i = 0$ for $2 \le i \le c$ and $\sum_{j \in V_1} w_j = p(C_1) + \sum_{i \ge 2} d_i = 0$. Since $U = U_1 \cdots U_c$ with $U_i \in W_{B_{V_i}}$ and the $V_i$ are disjoint, each $U_i$ is the restriction of $U$ to $V_i$ and is thus determined by $g$ and $m$. Moreover $U_i$ has underlying permutation $C_i$, support in $V_i$ and signed parity $0$. Applying Lemma~\ref{lem:cl-tree} to the cycle $C_i$, the reduced factorization $T_i$ and the element $U_i$ therefore determines the labels of $\underline{U}_i$ uniquely.

All in all we have shown that every full reflection factorization of $g \in W_{D_n}$ of length $m$ is Hurwitz equivalent to the same tuple, which only depends on $g$ and $m$ as well as the choices of the $r_i$ and $T_i$.
\end{proof}

\subsection{Property (F) for type $B$}

Given a long reflection $t_{ij}^a \in W_{B_n}$, we call $(i \, j)$ the \emph{edge} of the reflection.

\begin{lemma}
\label{lem:cl-short}
Let $\underline{t} = (t_1,\dots,t_m)$ be a tuple of reflections of $W_{B_n}$, let $p, q \in \{1, \dots, m\}$, $p \neq q$ be such that $t_p = \delta_i$ is short, $t_q = t^{a}_{ij}$ is long, 
and put $w := e_i + e_j$. Then
$\underline{t} \sim \underline{t}\,'$, where $\underline{t}\,'$ is given entrywise by
\[
  t'_k :=
  \begin{cases}
    \delta_j, & k = p,\\[2pt]
    t^{\,a+1}_{ij}, & k = q,\\[2pt]
    \delta_w\, t_k\, \delta_w, & \min(p,q) < k < \max(p,q),\\[2pt]
    t_k, & \text{otherwise.}
  \end{cases}
\]
Moreover:
\begin{enumerate}
\item[(a)] $\rho(t'_k) = \rho(t_k)$ for every $k$; in particular $t'_k$ is short if and only
if $t_k$ is short, and the long entries of $\underline{t}\,'$ carry, in the same positions and
the same order, the same edges as those of $\underline{t}$.
\item[(b)] $t'_k = t_k$ for every short entry $t_k$ with $k \ne p$; the entries that actually
change are $t_p$, $t_q$, and those long entries $t_k$ with $\min(p,q) < k < \max(p,q)$ and
$w_{k'} + w_{k''} = 1$, where $\{k',k''\}$ is the edge of $t_k$, and these change only in
their label.
\end{enumerate}
\end{lemma}

\begin{proof}
We shall use the following computations in $W_{B_n}$, each of them direct consequences of \eqref{eq:cl-pair-product} and \eqref{eq:cl-diagonal-conjugation}.

\begin{itemize}
    \item[(C1)] For $u \in \mathbb F_2^n$ and $(v,\pi) \in W_{B_n}$ we have $(v,\pi)^{-1}\delta_u(v,\pi) = \delta_{\pi^{-1}(u)}$. In particular
$t^{a}_{ij}\,\delta_i\,t^{a}_{ij} = \delta_j$.

\item[(C2)] By identity~(2) we have $\delta_u\, t^{b}_{kl}\, \delta_u = t^{\,b + u_k + u_l}_{kl}$. Thus conjugation by $\delta_u$ preserves the edge of a long reflection and changes at most its label.

\item[(C3)] By (C1) and (C2) we have the following Hurwitz equivalences.
\[
  (\delta_i,\,t^{a}_{ij}) \;\sim\; (t^{a}_{ij},\,\delta_j)
    \;\sim\; (\delta_j,\,t^{\,a+1}_{ij}),
  \qquad
  (t^{a}_{ij},\,\delta_i) \;\sim\; (\delta_j,\,t^{a}_{ij})
    \;\sim\; (t^{\,a+1}_{ij},\,\delta_j).
\]
Indeed, in both cases the second equivalence is (C2) with $u = e_j$, which shifts the label by $(e_j)_i + (e_j)_j = 1$.
\end{itemize}

\smallskip
We first consider the case $p < q$. Apply to $\underline{t}$ the braid
\[
  \beta := \bigl(\sigma_p\sigma_{p+1}\cdots\sigma_{q-2}\bigr)
           \;\cdot\; (\sigma_{q-1}\sigma_{q-1})
           \;\cdot\; \bigl(\sigma_{q-2}^{-1}\cdots\sigma_{p+1}^{-1}\sigma_p^{-1}\bigr),
\]
where the first and third blocks are empty if $q = p+1$. The third block moves $\delta_i$ from position $p$ to position $q-1$. The entries originally at positions $p+1,\dots,q-1$ are each conjugated by $\delta_i$. The tuple now has $(\delta_i, t^{a}_{ij})$ at positions $q-1,q$, so applying $\sigma_{q-1}\sigma_{q-1}$ yields $(\delta_j,\,t^{\,a+1}_{ij})$ there by (C3). Applying finally the first block of $\beta$, moves $\delta_j$ from position $q-1$ back to position $p$, returning the intermediate entries to their original positions $p+1,\dots,q-1$ and conjugating each of them by $\delta_j$.

Altogether, position $p$ carries $\delta_j$, position $q$ carries $t^{\,a+1}_{ij}$, every
entry outside $\{p,\dots,q\}$ is unchanged, and an entry $t_k$ with $p < k < q$ has been
replaced by
\[
  \delta_j\bigl(\delta_i\, t_k\, \delta_i\bigr)\delta_j
  = (\delta_j\delta_i)\, t_k\, (\delta_i\delta_j)
  = \delta_w\, t_k\, \delta_w.
\]
This is exactly $\underline{t}\,'$.

For the case $p > q$, apply the braid
\[
  \beta' := \bigl(\sigma_{p-1}^{-1}\sigma_{p-2}^{-1}\cdots\sigma_{q+1}^{-1}\bigr)
            \;\cdot\; \sigma_q^{-1}\sigma_q^{-1}
            \;\cdot\; \bigl(\sigma_{q+1}\sigma_{q+2}\cdots\sigma_{p-1}\bigr).
\]
and argue as in the case $p<q$.

It remains to prove (a) and (b). Since $\rho$ is a homomorphism and $\rho(\delta_u) = 1$ for every $u \in \mathbb{F}_2^n$, conjugation by $\delta_w$ does not change the underlying permutation of an entry. Moreover $\rho(\delta_j) = 1 = \rho(\delta_i)$ and $\rho(t^{\,a+1}_{ij}) = (i\,j) = \rho(t^{a}_{ij})$. Hence $\rho(t'_k) = \rho(t_k)$ for all $k$. As a reflection of $W_{B_n}$ is short precisely when its underlying permutation is trivial, the short reflections of $\underline{t}\,'$ occupy the same positions as those of $\underline{t}$, and the long reflections keep their edges, which shows (a).

For (b), let $t_k$ be short with $k \ne p$. If $k<p$ or $k>q$, then $t'_k = t_k$ by definition. If $p < k <q$, then $t_k = \delta_l$ for some $l$, and $t'_k = \delta_w \delta_l \delta_w = \delta_l = t_k$ by (C1). Finally, a long reflection $t_k = t^{b}_{k'k''}$ with $p<k<q$ satisfies $t'_k = t^{\,b + w_{k'} + w_{k''}}_{k'k''}$ by (C2), so it changes exactly when
$w_{k'} + w_{k''} = 1$, and then only in its label.
\end{proof}

\medskip
\noindent For a tuple $\underline{t}$ of reflections of $W_{B_n}$ we call
\[
G_{\mathrm{lg}}(\underline{t}) := \bigl([n],\ \{\,\{k',k''\} : t_k = t^{b}_{k'k''}
      \text{ a long reflection  for some } b \,\}\bigr)
\]
the \emph{long graph} of $\underline{t}$.

\begin{corollary}\label{cor:relocate}
Let $\underline{t}$ be a tuple of reflections of $W_{B_n}$ with $t_p = \delta_i$ short, and let $j \in [n]$ lie in the same connected component of $G_{\mathrm{lg}}(\underline{t})$ as $i$. Then $\underline{t} \sim \underline{t}\,''$ for a tuple $\underline{t}\,''$ with $t''_p = \delta_j$, with $t''_k = t_k$ for every short entry $t_k$, $k \ne p$, and with $\rho(t''_k) = \rho(t_k)$ for every $k$.
\end{corollary}

\begin{proof}
Choose a path $i = k_0, k_1, \dots, k_r = j$ in $G_{\mathrm{lg}}(\underline{t})$ and, for each $s$, a position $q_s$ whose entry is long reflection with edge $\{k_{s-1},k_s\}$. Note that $q_s \ne p$, as $t_p$ is short. We apply Lemma~\ref{lem:cl-short} $r$ times: to positions $p$ and $q_1$, then to positions $p$ and $q_2$, and so on, and finally to positions $p$ and $q_r$. Note that by part~(a) of the Lemma~\ref{lem:cl-short}, each application leaves the position and the edge of every long reflection unchanged, so after $s-1$ steps the position $q_{s}$ still carries a long reflection with edge $\{k_{s-1},k_s\}$, while the entry at position $p$ is $\delta_{k_{s-1}}$ by construction. After $r$ steps the entry at position $p$ is $\delta_{k_r} = \delta_j$. By part~(b) of Lemma~\ref{lem:cl-short} every short reflection other than the one at position $p$ is unchanged in each step, and by part~(a) the underlying permutations of all entries are unchanged in each step.
\end{proof}

\begin{lemma}[{\cite[Corollary~3.16]{DLM}}]\label{lem:parabolicB}
An element of $W_{B_n}$ is a parabolic quasi-Coxeter element if and only if at most one of its cycles is negative.
\end{lemma}

\begin{theorem}
\label{thm:cl-b-full}
Let $n \ge 2$ and $g \in W_{B_n}$. Two full reflection factorizations of $g$ in $W_{B_n}$ with the same number of short reflections and the same number of long reflections are Hurwitz equivalent.
\end{theorem}

\begin{proof}
Let $n \ge 2$.
We will show that two full reflection factorizations of $g \in W_{B_n}$ with the same number of short reflections and the same number of long reflections are Hurwitz equivalent by constructing a specific tuple and showing that each of the two is Hurwitz equivalent to this specific tuple. Throughout, $q$ denotes the common number of short reflections and $m$ the common number of long reflections of the two given tuples, so their common length is $q+m$.
We write $g = (v,\pi)$, let $C_1,\dots,C_c$ be the cycles of $\pi$ (fixed points
included), and put $V_i := \operatorname{supp}(C_i)$. 
Note that since $\rho(\delta_i) = 1$ and $\rho(t^a_{ij}) = (i\,j)$, a reflection of $W_{B_n}$ is short precisely when its image under $\rho$ is trivial.

The set of reflections $T_{B_n}$ of $W_{B_n}$ has exactly two conjugacy classes, one given by the short reflections, the other one given by the long reflections. So prescribing the numbers $q$ and $m$ is the same as prescribing the class-count vector. 

Let $\underline{t}$ be a full reflection factorization of $g$ with $q$ short and $m$ long reflections. All long reflections lie in the proper subgroup $W_{D_n}$, so a tuple of long reflections alone cannot generate $W_{B_n}$, hence $q \ge 1$. Moreover $\rho(\underline{t})$ generates $\rho(W_{B_n}) = \mathfrak S_n$,
so the underlying transpositions of the long reflections already generate $\mathfrak S_n$. Equivalently, by \cite[Lemma 3.10.2]{GR01}, the long graph $G_{\mathrm{lg}}(\underline{t})$ is connected.

Now, if $g$ has at most one negative cycle, then $g$ is a parabolic quasi-Coxeter element by Lemma~\ref{lem:parabolicB}. The assertion then follows by \cite[Theorem~18]{DL}.
Therefore we can assume that $g$ has at least two negative cycles. In particular $c \ge 2$. Fix once and for all representatives $r_i \in V_i$ and reduced transposition factorizations $T_i$ of $C_i$, and set $e := (r_1\,r_2)$ and $f_i := (r_1\,r_i)$ for $3 \le i \le c$. Since the long reflections of a full reflection factorization project to a length-$m$ transposition factorization of $\pi$, which generates $\mathfrak S_n$, Lemma~\ref{lem:cl-projection} gives
\[
  m = n + c - 2 + 2h, \qquad h := \tfrac12\,(m-n-c+2) \in \mathbb Z_{\ge 0}.
\]

\noindent Call a tuple $\underline{t}$ of reflections of $W_{B_n}$ with $q$ short and $m$ long reflections \emph{in normal position} if
\begin{equation}\label{eq:normalposB}
  \underline{t} = \bigl(
    \underbrace{\delta_{r_1},\dots,\delta_{r_1}}_{\text{short prefix, } q \text{ entries}};\;
    \underbrace{t^{a_1}_{r_1r_2},\dots,t^{a_{2h+2}}_{r_1r_2}}_{e\text{-lift}},\;
    \underbrace{t^{b_3}_{r_1r_3},t^{b'_3}_{r_1r_3},\dots,t^{b_c}_{r_1r_c},t^{b'_c}_{r_1r_c}}
      _{f\text{-lift}},\;
    \underline{U}_1,\dots,\underline{U}_c
  \bigr),
\end{equation}
where $\underline{U}_i$ is a tuple of $|V_i| - 1$ long reflections whose underlying transpositions are those of $T_i$, in the order prescribed by $T_i$. As in type $D_n$, we call the $e$-lift together with the $f$-lift the \emph{bridge}, the blocks $\underline{U}_i$ the \emph{tree part}, and the collection of all labels occurring in \eqref{eq:normalposB} the \emph{label vector}. The short prefix carries no labels. A tuple in normal position is determined by its label vector, since a long reflection is determined by its edge and its label.

Let $\underline{t}$ again be a full reflection factorization of $g$ in $W_{B_n}$ with $q$ short and $m$ long reflections. We bring $\underline{t}$ into normal position. To do so, we use the Hurwitz action. In particular this preserves the product, the length, and the numbers $q$ and $m$. We proceed in three moves.

(i) If a short reflection $\delta_l$ sits at position $k$ and the entry $x$ at position $k-1$ is long, then $\sigma_{k-1}(x,\delta_l)= (\delta_l, \delta_l x \delta_l)$, so the short reflection moves one position to the left and the long reflection it passes is conjugated, which preserves its edge and changes at most its label. Repeating this process, we may assume that the first $q$ entries are short and the last $m$ are long, which we call the \emph{long suffix}. The relative order within each group and the ordered list of edges of the long entries being those of $\underline{t}$.

(ii) The long suffix projects to a length-$m$ transposition factorization of $\pi$ generating $\mathfrak S_n$.
By Lemma~\ref{lem:cl-projection} there is a braid $\beta$, acting on the last $m$ strands only, such that $\beta$ applied to $\rho(\text{long suffix})$ is a factorization as given by \eqref{eq:nf}. By equivariance, applying $\beta$ to the long suffix leaves the short prefix untouched and puts the long suffix into the shape displayed in \eqref{eq:normalposB}.

(iii) As observed above, the long graph is connected. We can therefore apply Corollary~\ref{cor:relocate} to each short reflection. By doing so, we replace that short reflection by $\delta_{r_1}$, leave every other short reflection unchanged, and leave $\rho(\underline{t})$ unchanged entrywise. Hence this preserves both the position of the short prefix and the shape \eqref{eq:normalposB} of the long suffix, changing at most the labels of long reflections.

After (i)--(iii) the tuple is in normal position.

Analogous to the proof of Theorem~\ref{thm:cl-d-full}, for $\underline{t}$ as in \eqref{eq:normalposB} we put
\[
  q_e := |\{\mu : a_\mu = 1\}|,\quad
  d_2 := \sum_{\mu=1}^{2h+2} a_\mu \equiv q_e \!\!\pmod 2, \quad
  d_i := b_i + b'_i \ \ (3 \le i \le c),
\]
and let $S$, $D$, $U$ be the products of the short prefix, of the bridge and of the tree part, so that $S D U = g$. Exactly as in the proof of Theorem~\ref{thm:cl-d-full}, grouping the $e$-lift into $h+1$ pairs and applying identity~\eqref{eq:cl-pair-product} to each pair and to each pair $(t^{b_i}_{r_1r_i},t^{b'_i}_{r_1r_i})$ in the $f$-lift gives
\begin{equation}\label{eq:SD}
  D = \delta_{u},~ u = \sum_{i=2}^{c} d_i\,(e_{r_1}+e_{r_i}),
  \quad
  S = \delta_{r_1}^{\,q} = \delta_{\bar q\, e_{r_1}},~ \bar q := q \bmod 2 .
\end{equation}
Both $D$ and $S$ are involutions and they commute, so $SD = \delta_{s}$ with
$s = \bigl(\bar q + \sum_{i\ge2} d_i\bigr) e_{r_1} + \sum_{i \ge 2} d_i\, e_{r_i}$, and
$U = (SD)^{-1} g = D S g$.

We write $U_i$ for the product of $\underline{U}_i$. By construction we have $U_i \in W_{B_{V_i}}$ with $\rho(U_i) = C_i$. Since $p_{V_i}$ (see \eqref{eq:DefpV}) vanishes on every long reflection, we have $p_{V_i}(U_i) = 0$.
As the $V_i$ are pairwise disjoint,
$U = U_1\cdots U_c = (w,\pi)$ with $\sum_{j \in V_i} w_j = 0$ for every $i$. From
$g = (SD)\,U = (s + w,\pi)$ we get $v = s + w$, hence, summing over $V_i$,
\[
  p(C_i) = \sum_{j \in V_i} v_j = \sum_{j \in V_i} s_j .
\]
By \eqref{eq:SD}, for $i \ge 2$ the only coordinate of $s$ inside $V_i$ is the one at $r_i$,
and for $i = 1$ it is the one at $r_1$. Therefore
\begin{equation}\label{eq:diB}
  d_i = p(C_i) \quad (2 \le i \le c),
  \qquad\text{and}\qquad
  \bar q + \sum_{i=2}^{c} d_i = p(C_1).
\end{equation}
The second identity is no extra condition. Indeed, the total sign parity $\varepsilon \colon W_{B_n} \to \mathbb F_2$, $(v,\pi) \mapsto \sum_{j} v_j$, is a homomorphism
vanishing on every long reflection and taking the value $1$ on every short reflection, so
evaluating it on $\underline{t}$ and on $g$ gives
\[
  q \;\equiv\; \sum_{i=1}^{c} p(C_i) \pmod 2 ,
\]
which together with the first part of \eqref{eq:diB} is precisely the second part.

As $q \geq 1$, the short prefix is nonempty, and we put $B := (\delta_{r_1})$ (the block consisting of its first entry) and $b := \delta_{e_{r_1}}$. Every bridge entry has an edge incident to $r_1$, so identity~\eqref{eq:cl-diagonal-conjugation} yields
\[
  \bigl(t^{a}_{r_1 r_i}\bigr)^{\,b} = t^{\,a + 1}_{r_1 r_i} \qquad (2 \le i \le c).
\]
Also note that two long reflections with the same edge commute, so adjacent entries of the $e$-lift and the two entries of the pair $(t^{b_i}_{r_1r_i},t^{b'_i}_{r_1r_i})$ in the $f$-lift may be interchanged by a Hurwitz move, all other entries remaining unchanged.

Analogous to the proof of Theorem~\ref{thm:cl-d-full} (with $B$ as chosen before), we can now permute the $e$-lift and apply Lemma~\ref{lem:cl-pair} until $q_e =d_2 \in \{0,1\}$ and the labels of the $e$-lift are given by $(0,d_2,0,\dots,0)$. Also analogous to the proof of Theorem~\ref{thm:cl-d-full} and by applying Lemma~\ref{lem:cl-pair} we can assure by suitable Hurwitz moves that the labels $(b_i, b_i')$ of the pair $(t^{b_i}_{r_1r_i},t^{b'_i}_{r_1r_i})$ in the $f$-lift are $(0,1)$ if $d_i = 1$ and $(0,0)$ if $d_i = 0$. All operations of this step preserve $\rho(\underline{t})$ entrywise and therefore the normal position.

At this point the short prefix consists of $q$ copies of $\delta_{r_1}$ and the bridge has labels
\[
  (a_1,\dots,a_{2h+2}) = (0,d_2,0,\dots,0), \qquad (b_i,b'_i) = (0,d_i)\ \ (3 \le i \le c),
\]
which by \eqref{eq:diB} depends only on $g$, $q$ and $m$. Thus $S$ and $D$ depend only on $g$, $q$ and $m$ by \eqref{eq:SD}, and hence so does $U = DSg$. As shown above, the sign vector $w$ of $U = (w,\pi)$ satisfies $\sum_{j \in V_i} w_j = 0$ for every $i$. Since $U = U_1\cdots U_c$ with $U_i \in W_{B_{V_i}}$ and the $V_i$ disjoint, each $U_i$ is the restriction of $U$ to $V_i$ and is therefore determined by $g$, $q$ and $m$, has underlying permutation $C_i$, support in $V_i$ and signed parity $0$. Applying Lemma~\ref{lem:cl-tree} to the cycle $C_i$, the fixed reduced factorization $T_i$ and the element $U_i$ determines the labels of $\underline{U}_i$ uniquely.

We have thus shown that every full reflection factorization of $g \in W_{B_n}$ consisting of $q$ short and $m$ long reflections is Hurwitz equivalent to one and the same tuple in normal position, whose labels only depend on $g$, $q$, $m$ and the choices of the $r_i$ and $T_i$. Any two full reflection factorizations of $g$ consisting of $q$ short and $m$ long reflections are therefore Hurwitz equivalent.
\end{proof}

\subsection{Proof of the main theorem}

\begin{lemma}\label{lem:reduction}
Let $W$ be a finite Coxeter group. If every irreducible reflection subgroup of $W$ has property (F), then Conjecture~\ref{conjecture} holds for $W$.
\end{lemma}

\begin{proof}
Necessity follows directly from the definition of the Hurwitz action. For sufficiency, let $\underline t, \underline t'$ be reflection factorizations of $g \in W$ with $\langle \underline t\rangle = \langle \underline t'\rangle = H$ and with the same multiset of $H$-conjugacy classes. The reflection subgroup $H$ is itself a finite Coxeter group with set of reflections $T \cap H$. Let $H = H_1 \times \dots \times H_s$ be its decomposition into irreducible components. Every reflection of $H$ lies in exactly one $H_i$. Reflections in distinct components commute, and a Hurwitz move applied to two commuting adjacent entries interchanges them without changing either entry. Hence both tuples can be sorted into blocks according to the components, leaving all entries unchanged. By hypothesis the two tuples have the same multiset of $H$-conjugacy classes, so for each $i$ the two $i$-th blocks have the same length and the same multiset of $H_i$-conjugacy classes, each of them is full in $H_i$, and both have product the $H_i$-component of $g$. Applying property (F) for $H_i$ inside each block and composing the resulting braids with the two sorting braids gives $\underline t \sim \underline t'$.
\end{proof}

\noindent The following result is well-known.
\begin{lemma}
\label{lem:cl-subgroups}
A reflection subgroup of a finite Coxeter group of type $A$, $B$ or $D$ is, as a Coxeter group, a direct product of Coxeter groups of types $A$, $B$ and $D$. A reflection subgroup of $W_{D_n}$ has only factors of types $A$ and $D$.
\end{lemma}

\medskip
\begin{proof}[\textbf{Proof of Theorem~\ref{thm:main}}]
Let $W = W_1 \times \dots \times W_s$ be the decomposition of $W$ into irreducible components, so that each $W_i$ is of type $A$, $B$ or $D$, and let $K$ be an irreducible reflection subgroup of $W$. Every reflection of $W$ lies in exactly one factor $W_i$. Hence $K$ is the direct product of the subgroups generated by its reflections in the individual factors, and since $K$ is irreducible, $K$ is an irreducible reflection subgroup of some $W_i$. By Lemma~\ref{lem:cl-subgroups} the group $K$ is of type $A$, $B$ or $D$. Property~(F) holds in type $A$ by Lemma~\ref{lem:cl-symmetric}, in type $B$ by Theorem~\ref{thm:cl-b-full} and in type $D$ by Theorem~\ref{thm:cl-d-full}. Now apply Lemma~\ref{lem:reduction}.
\end{proof}

\bibliography{hurwitz}

\providecommand{\bysame}{\leavevmode\hbox to3em{\hrulefill}\thinspace}
\providecommand{\MR}{\relax\ifhmode\unskip\space\fi MR }
\providecommand{\MRhref}[2]{%
  \href{http://www.ams.org/mathscinet-getitem?mr=#1}{#2}
}
\providecommand{\href}[2]{#2}
\begin{thebibliography}{10}

\bibitem{BGRW}
Barbara Baumeister, Thomas Gobet, Kieran Roberts, and Patrick Wegener, \emph{On the {Hurwitz} action in finite {Coxeter} groups}, J. Group Theory \textbf{20} (2017), no.~1, 103--131.

\bibitem{Berger}
Emily Berger, \emph{Hurwitz equivalence in dihedral groups.}, Electron. J. Comb. \textbf{18} (2011), no.~1, 16.

\bibitem{Bessis}
David Bessis, \emph{The dual braid monoid.}, Ann. Sci. {\'E}c. Norm. Sup{\'e}r. (4) \textbf{36} (2003), no.~5, 647--683 (English).

\bibitem{DL}
Theo Douvropoulos and Joel~Brewster Lewis, \emph{Hurwitz orbits on reflection factorizations of parabolic quasi-{Coxeter} elements}, Electron. J. Comb. \textbf{31} (2024), no.~1, 17.

\bibitem{DLM}
Theo Douvropoulos, Joel~Brewster Lewis, and Alejandro~H. Morales, \emph{Hurwitz numbers for reflection groups. {II}: {Parabolic} quasi-{Coxeter} elements}, J. Algebra \textbf{641} (2024), 648--715.

\bibitem{GR01}
Chris Godsil and Gordon Royle, \emph{Algebraic graph theory}, Grad. Texts Math., vol. 207, New York, NY: Springer, 2001.

\bibitem{Humphreys}
James~E. Humphreys, \emph{Reflection groups and {Coxeter} groups}, Camb. Stud. Adv. Math., vol.~29, Cambridge etc.: Cambridge University Press, 1990.

\bibitem{KL88}
Paul Kluitmann, \emph{Hurwitz action and finite quotients of braid groups}, Braids, {AMS}-{IMS}-{SIAM} {Jt}. {Summer} {Res}. {Conf}., {Santa} {Cruz}/{Calif}. 1986, {Contemp}. {Math}. 78, 299-325, 1988.

\bibitem{Lewis}
Joel~Brewster Lewis, \emph{A note on the {Hurwitz} action on reflection factorizations of {Coxeter} elements in complex reflection groups}, Electron. J. Comb. \textbf{27} (2020), no.~2, 14, Id/No p2.54.

\bibitem{LR}
Joel~Brewster Lewis and Victor Reiner, \emph{Circuits and {Hurwitz} action in finite root systems}, New York J. Math. \textbf{22} (2016), 1457--1486.

\bibitem{WY2}
Patrick Wegener and Sophiane Yahiatene, \emph{A note on non-reduced reflection factorizations of {Coxeter} elements}, Algebr. Comb. \textbf{3} (2020), no.~2, 465--469.

\bibitem{WY}
\bysame, \emph{Reflection factorizations and quasi-{Coxeter} elements}, J. Comb. Algebra \textbf{7} (2023), no.~1-2, 127--157.

\end{thebibliography}
\bibliographystyle{amsplain}

\vspace{2em}
\noindent
\begin{minipage}{\linewidth}
\raggedright
\small{Patrick Wegener}\\
\small Philipps-Universit\"at Marburg, Germany\\
\small \textit{E-mail address:} \texttt{wegenerp@staff.uni-marburg.de}
\end{minipage}

\end{document}